\documentclass[12pt]{amsart}

\usepackage{amsmath, amsthm, amssymb}

\usepackage{url} 
\usepackage{graphicx}
\usepackage{xcolor}
\usepackage{moreverb}

\usepackage{fancyvrb}

\def\r{\mathbb R}

\def\s{\mathbb S}
\def\h{\mathbb H}
\def\e{\mathbb E}
\def\v{\mathcal V}
\def\w{\mathcal W}
\def\x{\mathfrak X}

\usepackage{listings}
 \lstnewenvironment{code}[1][]
  {\lstset{#1}}
  {}

\newtheorem{theorem}{Theorem}[section]
 \newtheorem{proposition}[theorem]{Proposition}
 \newtheorem{corollary}[theorem]{Corollary}
\theoremstyle{definition}
\newtheorem{definition}[theorem]{Definition}
\newtheorem{example}[theorem]{Example}
\newtheorem{remark}[theorem]{Remark}

\begin{document}

\title[Hypersurfaces with torse-forming axes]{Hypersurfaces in Riemannian manifolds with torse-forming axes}

\author{Muhittin Evren Aydin$^1$}
\address{$^1$Department of Mathematics, Faculty of Science, Firat University, Elazig,  23200 Turkey
\newline
ORCID: 0000-0001-9337-8165}
\email{meaydin@firat.edu.tr}
\author{ Adela Mihai$^2$}
 \address{$^2$Technical University of Civil Engineering Bucharest,
Department of Mathematics and Computer Science, 020396, Bucharest, Romania
and Transilvania University of Bra\c{s}ov, Interdisciplinary Doctoral
School, 500036, Bra\c{s}ov, Romania
\newline
ORCID: 0000-0003-2033-8394}

 \email{adela.mihai@utcb.ro, adela.mihai@unitbv.ro}
\author{ Cihan Özgür$^3$}
 \address{$^3$Department of Mathematics, Izmir Democracy University, 35140, Karabaglar, Izmir, Turkey
 \newline
ORCID: 0000-0002-4579-7151}
 \email{cihan.ozgur@idu.edu.tr}

\begin{abstract}
In this paper, we study orientable hypersurfaces $N$ in Riemannian manifolds $(M,\langle , \rangle)$ for which the inner product $\langle U, \v \rangle$ is constant, where $U$ is the unit normal vector field to $N$ and $\v$ is a globally defined torse-forming vector field on $M$, called the axis of $N$. When $\v$ is a unit torse-forming vector field, $N$ becomes a constant angle hypersurface with axis $\v$, and we classify such hypersufaces. After that, the case when $\v$ is a torqued vector field is considered and a corresponding classification is given.
\end{abstract}
\maketitle

\section{Introduction} \label{intro}

One of the major problems in classical differential geomety is to determine the geometric properties of constant angle submanifolds, which has been widely studied in the literature. Without attempting to give a complete list, we refer to \cite{bar,cd,crs0,dgd,dmvv,dsr,gpr,lo,loy3,mun,mun1,mun2}.  More clearly, let $M$ be a Riemannian manifold, and let $N\subset M$ be a submanifold. Once a vector field $\v$ is fixed in the ambient space $M$, a constant angle submanifold $N$ is defined as a submanifold of $M$ for which the angle between $\v$ and $U$ is constant on $N$, where $U$ denotes the tangent or principal normal vector field of $N$ in the case of a curve, an arbitrary vector field in the first normal space of $N$ for a submanifold, or the unit normal vector field of $N$ in the case of a hypersurface. 

In the study of such problems, one naturally encounters another fundamental question: the global existence of the vector field $\v$ in the ambient space $M$, which governs its geometry and topology. Hence, once the type of the vector field $\v$ is fixed, the investigation of constant angle submanifolds also leads to consider the problem of determining those ambient spaces admitting such a vector field $\v$. 

Regarding the types of vector fields, we recall the notion of a torse-forming vector field $\v$ on $M$, introduced by Yano \cite{ya1,ya0}. Let $\nabla^0$ be the Levi-Civita connection on $M$. A vector field $\v$ is said to be {\it torse-forming} if it satisfies
\begin{equation}\label{int-tor}
\nabla^0_X\v=fX+\omega(X) \v, \quad X\in \mathfrak{X}(M),
\end{equation}
where $f$ is a smooth function on $M$ and $\omega$ is a $1$-form, called the conformal scalar (or potential function) and the generating form of $\v$, respectively (see also \cite{bo1}). If the conformal scalar $f$ and the generating form $\omega$ is nowhere zero, then the torse-forming vector field $\v$ is said to be proper.

The existence of torse-forming vector fields on contact geomety, generalized Robertson-Walker space times, twisted Lorentizan manifolds, and on spheres was investigated in \cite{c0,id,mam0,mam1,mih0}. Moreover, the authors in \cite{c00} established an interesting link between the rotational hypersurfaces immersed in $\e^m$ with an axis passing through the origin and the torse-forming vector fields: the tangential component of the immersion always defines a proper torse-forming vector field.

The vector field $\w$ dual to the $1$-form $\omega$ is called {\it generative} of $\v$. We highlight two subclasses of torse-forming vector fields, called torqued and anti-torqued vector fields \cite{crs3,dan}. A torse-forming vector field is called a {\it torqued} (respectively {\it anti-torqued}) vector field if $\w$ is perpendicular (respectively parallel) to $\v$. See also \cite{cs,na}. Similar to the torse-forming vector fields, the existence of these vector fields on the Euclidean space $\e^m$, the hyperbolic space $\h^m$, and the sphere $\s^m$ were studied in \cite{amc,dan,dta}. 

Torqued and anti-torqued vector fields are used as tools to characterize manifolds and submanifolds. For example, in \cite{amc,crs3} rectifying submanifolds in Riemannian manifolds were characterized. The authors in \cite{ymy} investigated submanifolds of Kenmotsu manifolds admitting torqued vector fields. Recently, curves and hypersurfaces in real space forms have been classified by using concircular vector fields, which are particular cases of torse-forming vector fields with vanishing $\omega$ \cite{loy4,loy5}. Recall that a concircular hypersurface $N$ is an orientable hypersurface in a real space form $M$ such that $\langle \v,U\rangle$ is a constant function $\vartheta$ on $N$, where $U$ is the unit normal vector field on $N$ and $\langle ,\rangle$ is the metric tensor of $M$.  

Motivated by the studies \cite{loy4,loy5}, we consider hypersurfaces $N$ in Riemannian manifolds $M$ satisfying 
\begin{equation}\label{intro1}
\langle \v,U\rangle=\vartheta, \quad \vartheta \in \r,
\end{equation}
where $\v$ is a torse-forming vector field on $M$ and $U$ is the unit normal vector field on $N$. We call $\v$ the {\it axis} of $N$. In particular, we are interested in two cases: the case $\v$ is a unit torse-forming vector field and the case $\v$ is a torqued vector field. Recall that every unit torse-forming vector field reduces to a unit anti-torqued vector field (see \cite[Proposition 3.6]{amc}). Moreover, as long as $\v$ is a unit vector field, we have $\vartheta \in [-1,1]$, and the condition that $\langle \v,U\rangle=\vartheta$ implies that the angle between $\v$ and $U$ is constant at every point of $N$. In this case, $N$ is known as a helix hypersurface, a constant slope hypersurface or a constant angle hypersurface with axis $\v$. 

We would like to emphasize two main differences between our study and the cited works. First, in the cited papers at the beginning of this section, the ambient space and the corresponding axis $\v$ were fixed in advance, and the results and examples were obtained with respect to that particular setting. In contrast, in our study the ambient space is an arbitrary Riemannian manifold and the axis $\v$ can be either a unit torse-forming or a torqued vector field. Hence, depending on the choice of the ambient space, different examples of such axes may arise. For instance, in a real space form, more than one concircular axis can exist (see \cite[Theorem 1]{loy5}).

Second, a torse-forming vector field on $M$, for example when it is particularly a torqued vector field, is not necessarily of constant length (see \cite[Proposition 3.2]{amc}). Hence, in such a case, the hypersurface $N$ lacks the constant angle property with respect to the axis $\v$.

The organization of the paper is the following. In Section \ref{pre}, we provide the basic formulas for hypersurfaces in Riemannian manifolds. In Section \ref{anti}, assuming the axis to be a unit torse-forming vector field, we first classify constant angle surfaces in three dimensional Riemannian manifolds (see Theorem \ref{ts1}). Then, under a particular assumption, this result is extended to higher dimensions (see Theorem \ref{ths1}). Finally, in Section \ref{torqued}, we characterize the hypersurfaces whose axis is torqued (see Theorems \ref{ths10} and \ref{ths11}). Non trivial examples corresponding to the obtained results are provided in both sections.

\section{Preliminaries} \label{pre}

Let $(M,\langle,\rangle)$ be a Riemannian manifold with $\mbox{dim}(M)=n+1$ and $\nabla^0$ the Levi-Civita connection on $M$. The Riemannian curvature tensor of $\nabla^0$ is defined as
$$
R^0(X,Y)Z=\nabla^0_X\nabla^0_YZ-\nabla^0_Y\nabla^0_XZ-\nabla^0_{[X,Y]}Z, \quad X,Y,Z \in \x(M).
$$
Denote by $u,v\in T_pM$ two linearly independent vectors tangent to $M$ at $p \in M$ and by $\pi=\text{Span}\{u,v\}$ a plane section. Then, the sectional curvature of $\pi$ is defined by
$$
K^0(\pi)=\frac{\langle(R(u,v)v,u\rangle}{\langle u,u\rangle \langle v,v\rangle-\langle u,v\rangle^2}.
$$

$M$ is called a {\it space form} if its sectional curvature $K^0$ is a constant, i.e. $K^0=c\in \r$ for all plane sections $\pi$ and points $p \in M$. In this case, the curvature tensor $R^0$ becomes
$$
R^0(X,Y)Z=c\{\langle Y,Z\rangle X-\langle X,Z\rangle Y\}, \quad  X,Y,Z \in \x(M).
$$

Let $\phi:N\to M$ be a smooth isometric immersion of codimension $1$. For a point $p\in N$, let $U:=(T_pN)^\perp$ be the unit normal vector field at $p$, that is, $|U|=1$, where $|\cdot|$ is the induced norm by the Riemannian metric $\langle , \rangle$. Then the shape operator $A:T_pN \to T_pN$ is the symmetric $(1,1)$-tensor field defined by $A(X)=-\nabla^0_X U$, $X\in \x(N)$. Since $A$ is symmetric, there exists an orhonormal frame of eigenvectors $\{e_1,...,e_n\}$ of $T_pN$ with corresponding real eigenvalues $\kappa_1,...,\kappa_n$. 

Assume that both $M$ and $N$ are orientable and oriented, i.e., orientations are chosen on each of them. Then the vectors $e_i$ and the scalars $\kappa_i$ are called the principal directions and principal curvatures of $\phi$, respectively. Moreover, the quantity 
$$H=\frac{1}{n}(\kappa_1+\cdots+\kappa_m)$$ is called the mean curvature of $\phi$ \cite{car}.

The Gauss formula is given by
$$
\nabla^0_XY=\nabla_XY+\langle A(X),Y\rangle U, \quad X,Y \in\x(N),
$$
where $\nabla$ is the induced Levi-Civita connection on $N$. A hypersurface is called totally geodesic if $A=0$, and totally umbilical if $A=\delta \mbox{Id}$, for a smooth function $\delta$ on $N$. Also, the Codazzi equation is written by
$$
(R^0(X,Y)U)^{\top}=(\nabla_XA)Y-(\nabla_YA)X, \quad X,Y \in \x(N).
$$
If, in particular, $M$ is a space form, then the Codazzi equation reduces to 
$$(\nabla_XA)Y=(\nabla_YA)X, \quad X,Y \in \x(N).$$

Let $M_1\times M_2$ be the product of two Riemannian manifolds $(M_1,\langle,\rangle_1)$ and $(M_2,\langle,\rangle_2)$. Consider the natural projections $\pi_i : M_1 \times M_2 \to M_i,$ $i=1,2
$. A twisted product $M_1 \times_{\lambda} M_2$ is a Riemannian manifold $M_1 \times M_2$ equipped with the metric tensor $$\langle, \rangle=\langle, \rangle_1+\lambda^2\langle, \rangle_2
,$$ where $\lambda$ is a smooth function on $M$. If, in particular, the function $\lambda$ depends only on $M_1$ then $M_1 \times_{\lambda} M_2$ is called a  warped product and $\lambda$ warping function. Let $M=M_1 \times_{\lambda} M_2$. Then, $M_1$ and $M_2$ are called the base and the fiber of $M$, respectively \cite{cbook}.

The leaves $M_1 \times \{ q\} = \pi_2^{-1}(q)$ and the fibers $\{p \} \times M_2 = \pi_1^{-1}(p)$ are submanifolds of $M$. A vector tangent to leaves (respectively, to the fibers) is called  horizontal (respectively vertical). We denote by $\mathcal{L}(M_1)$ (respectively, $\mathcal{L}(M_2)$) the set of all horizontal (respectively vertical) lifts.

We recall the following useful result.

\begin{proposition} \cite{cbook} \label{warp-con}
Let $M=M_1 \times_{\lambda} M_2$ be a warped product, and let $\nabla^0$ and $\nabla^i$ be the Levi-Civita connections on $M$ and $M_i$ $(i=1,2)$, respectively. Let also $X_i,Y_i \in \mathcal{L}(M_i)$. Then,
\begin{enumerate}
\item[(i)] $\nabla^0_{X_1}Y_1 \in \mathcal{L}(M_1) $ is the lift of $\nabla^1_{X_1}Y_1$ on $M_1$.
\item[(ii)] $\nabla^0_{X_1}X_2 =\nabla^0_{X_2}X_1 =(X_1 \log \lambda)X_2$.
\item[(iii)] $\textnormal{tan}(\nabla^0_{X_2}Y_2)$ is the lift of $\nabla^2_{X_2}Y_2$ on $M_2$.
\item[(iv)] $\textnormal{nor}(\nabla^0_{X_2}Y_2)=-\frac{\langle X_2,Y_2\rangle}{\lambda}\nabla^0\lambda$.
\end{enumerate}
\end{proposition}

\section{Hypersurfaces with anti-torqued axis} \label{anti}

In this section, we consider the hypersurfaces in Riemannian manifolds whose axis is a unit torse-forming vector field. As emphasized in the Introduction, every unit torse-forming vector field is a unit anti-torqued vector field. Recall  the notion of an anti-torqued vector field $\v$ on a Riemannian manifold $M$. 

Let $\v$ be a torse-forming vector field defined by \eqref{int-tor}. Denote by $f$ and $\omega$ its conformal scalar and generating form, respectively. Let $\w$ be the generative of $\v$, that is, the vector field dual to the $1$-form $\omega$. The vector field $\v$ is called anti-torqued if $\w=-f\v$. Equivalently, this condition can be written as $\omega=-f\nu$, where $\nu$ is the $1$-form dual to $\v$. Then, from \eqref{int-tor} it follows that
$$
\nabla^0_X\v=f(X- \langle X,\v\rangle\v), \quad X\in\x(M).
$$
By its definition, the vector field $\v$ is not parallel and we assume that $f$ is nowhere zero on $M$.

Consider a smooth immersion $\phi:N\to M$ of an orientable hypersurface $N$ into $M$. Let $U$ be its unit normal vector field and $\v$ a unit anti-torqued vector field on $M$. In this section, we consider hypersurfaces satisfying 
\begin{equation}
\langle U, \v \rangle = \cos \theta, \quad \theta \in [0,\pi]. \label{def31}
\end{equation}
Notice that an anti-torqued vector field always can be assumed to be unitary (see \cite{dan}). 

Denote by $\v^{\top}$ the tangential projection of $\v$ on $N$. If $\v^{\top}= 0$, then the hypersurface $N$ is totally umbilical, see \cite[Corollary 4.2]{amc1}. 

\begin{definition}
Let $\v$ be a unit anti-torqued vector field globally defined on a Riemannian manifold $M$. An orientable hypersurface $N\subset M$ is called a {\it constant angle hypersurface with anti-torqued axis} $\v$ if $\v^{\top}$ is nonzero on $N$ and condition \eqref{def31} is satisfied.
\end{definition}

Assume that $N$ is a constant angle hypersurface in $M$ whose axis is a unit anti-torqued vector field $\v$. Hence, we set
$$
T=\frac{\v^{\top}}{|\v^{\top}|}.
$$ 
Then we have the following orthogonal decomposition of $\v$ in its tangential and normal components with respect to $N$
\begin{equation}\label{decomp}
\v=\sin \theta T +\cos \theta U, \quad \theta \in [0,\pi]. 
\end{equation}

We begin with a useful result that will be needed later in this section.

\begin{proposition} \label{p41}
Let $N$ be a constant angle hypersurface in a Riemannian manifold with anti-torqued axis $\v$. Then we have
\begin{eqnarray}
\sin \theta \nabla_X T -\cos \theta A(X)&=&fX-f\sin^2 \theta\langle X,T\rangle T, \label{31}\\ 
\langle A(X),T\rangle&=&-f\cos\theta\langle X,T\rangle,  \label{32}\\
A(T) &=&-f\cos\theta T, \label{33}\\
\nabla_T T&=&0, \label{34}
\end{eqnarray}
where $f$ is the conformal scalar of $\v$.
\end{proposition}
\begin{proof}
By the assumption, we have decomposition \eqref{decomp}. By using the Gauss formula, we derive
$$
\nabla^0_X \v =\sin \theta \nabla_X T -\cos \theta A(X)+\sin \theta \langle A(X),T\rangle U, \quad X\in  \mathfrak{X}(N).
$$
Because $\v$ is an anti-torqued vector field, it follows
$$
\nabla^0_X \v =fX-f\sin \theta\langle X,T\rangle (\sin \theta T +\cos \theta U), \quad X\in  \mathfrak{X}(N).
$$
Comparing the tangent and normal components, we obtain \eqref{31} and \eqref{32}. The proofs of \eqref{33} and \eqref{34} follow directly from \eqref{31} and \eqref{32}.
\end{proof}

From now on, we distinguish two cases: $\mbox{dim}(M)=3$ and $\mbox{dim}(M)>3$. This is essential because, in \eqref{decomp}, the dimension of the orthogonal complement $T^{\perp} \subset \x(N)$ is either $1$ or greather than $1$, which effects our arguments.

\subsection{Case $\mbox{dim}(M)=3$} Denote by $M^3$ a $3$-dimensional Riemannian manifold. In the following result, we have a classification of constant angle surfaces in $M^3$.

\begin{theorem} \label{ts1}
Let $N$ be a constant angle surface in a complete Riemannian manifold $ M^3$ with anti-torqued axis $\v$. Then it is locally either a ruled surface or admits a local coordinate system $(s,t)$ whose metric takes the form 
\begin{equation}
\langle,\rangle_N=ds^2+\lambda^2dt^2, \quad (\log \lambda)_s=\frac{\cos\theta\kappa_2+f}{\sin\theta}, \label{tw}
\end{equation}
where $f$ is the conformal scalar of $\v$ and $\kappa_2$ is the principal curvature along $ \partial_t$.
\end{theorem}
\begin{proof}
Consider the orthonormal frame of principal directions $\{e_1,e_2\}$ of $\x(N)$ such that $T=e_1$ and $A(e_i)=\kappa_ie_i$. Then, using \eqref{33} and \eqref{34} we have $\kappa_1=-f\cos\theta $ and $\nabla_{e_1}e_1 =0.$ Hence, the shape operator takes the form
$$
A=\begin{pmatrix}
-f\cos\theta&0  \\ 
0& \kappa_2
\end{pmatrix}.%
$$
If $\cos\theta=0$, then by the Gauss formula we derive $\nabla^0_{e_1}e_1=0$, which means that the integral curves of $e_1$ are the geodesics of the ambient space $M^3$. In other words, the integral curves of $\mbox{span}\{e_1\}$ is a foliation of $N$ generated by the ambient geodescis; hence $N$ is a ruled surface in $M^3$ (see \cite{lrs}). 

Assume now that $\cos\theta\neq0$. Writing $X=e_2$ in \eqref{31} yields $\nabla_{e_2}e_1 =(\frac{\cos\theta\kappa_2 +f}{\sin\theta})e_2. $ From $\nabla_{e_1}e_1 =0$, it follows that $0=e_1 \langle e_1,e_2\rangle=\langle e_1,\nabla_{e_1}e_2 \rangle.$ Since $\nabla_{e_1}e_2$ is perpendicular to both $e_1$ and $e_2$, we have $\nabla_{e_1}e_2=0$. Similarly, from $0=e_2 \langle e_1,e_2\rangle$, we obtain $\nabla_{e_2}e_2 =-(\frac{\cos\theta\kappa_2 +f}{\sin\theta})e_1. $ Consequently, we conclude
\begin{equation}
\nabla_{e_1}e_1 =\nabla_{e_1}e_2=0, \quad \nabla_{e_2}e_1 =\sigma e_2, \quad \nabla_{e_2}e_2 =-\sigma e_1, \quad \sigma=\frac{\cos\theta\kappa_2 +f}{\sin\theta}. \label{s-1}
\end{equation}
Since the integral curves of $e_1$ are geodesics of $N$ we can choose a coordinate system $(s,t)$ on $N$ such that 
\begin{equation*} 
e_1=\partial_s:=\frac{\partial }{\partial s}, \quad \partial_t=\lambda(s,t)e_2,
\end{equation*}
for a smooth function $\lambda$ on $N$. Moreover, since $\nabla_{\partial_s}\partial_t=\nabla_{\partial_t}\partial_s$, after using \eqref{s-1} we derive
$$
\partial_s\langle \partial_t,\partial_t\rangle=2\langle \nabla_{\partial_t}\partial_s,\partial_t\rangle,
$$
which yields $(\lambda^2)_s=2\lambda^2\sigma$. Equivalently, we obtain $(\log \lambda)_s=\sigma$, completing the proof.
\end{proof}

In what follows, we provide an example of a constant angle surface with anti-torqued axis $\v$ whose metric is given by \eqref{tw} in the case the ambient space is a warped product $M^3=I\times_p F$, with $I\subset \r$ a real interval and $F$ a $2$-dimensional Riemannian manifold. For this, recall that if $u$ is a coordinate on $I$, then $\v=\partial_u$ is a unit anti-torqued vector field on $M^3$ whose conformal scalar is $f=(\log p)'$ and generating form $\omega=-f\nu$, where $\nu$ is the dual $1$-form of the lift to $M^3$ of $\partial_u$ (see \cite[Theorem 3.1]{amc1}).

\begin{example}\label{wapr}
Let $M^3=I\times_p \r^2$ be a $3$-dimensional warped product endowed with the metric 
$$
\langle , \rangle=du^2+p^2(u)(dx^2+dy^2),
$$ 
where $(x,y)$ are the canonical coordinate system of $\r^2$. The constant angle surfaces in $M^3$ with axis $\v=\partial_u$ were classified in \cite{dmvv}. Assume that $p(u)=u$. We consider the surface $N$ in $M^3$ parametrized by 
$$
\phi(s,t)=(\sin \theta s,\cot\theta \log s,t), \quad s>0.
$$
Its unit normal vector field is $U=(\cos\theta,-\frac1s,0)$ and hence $\langle U,\v \rangle=\cos\theta$. The first fundamental form of $N$ is given by
\begin{equation}
\langle , \rangle_N=ds^2+\lambda^2dt^2, \quad \lambda(s)=\sin\theta s, \label{43}
\end{equation}
which is a warped metric. Also, $N$ is a totally umbilical surface whose principal curvatures are
$$
\kappa_1(s)=\kappa_2(s)=-\frac{\cos\theta}{u\circ \phi(s,t)}=-\frac{\cot\theta}{s}=-\cos\theta f(s),
$$
where $f(s)=(f\circ\phi)(s,t)=\frac{1}{\sin \theta s}$. It is easy to see that the warping function $\lambda$ in \eqref{43} satisfies the condition in \eqref{tw}.
\end{example}

It is worth investigating the particular case when $M^3$ is one of the standard models of a complete simply-connected space form. Notice that the global existence of the anti-torqued vector field $\v$ implies that $M^3$ is either $\r^3$ or $\h^3$.
Using the Codazzi equation and \eqref{s-1}, we have
$$
0=(\nabla_{e_1}A)(e_2)-(\nabla_{e_2}A)(e_1)=-e_2(\kappa_1)e_1+(e_1(\kappa_2)+(\kappa_2-\kappa_1)\sigma)e_2,
$$
or equivalently
$$
e_2(\kappa_1)=0, \quad e_1(\kappa_2)=(\kappa_1-\kappa_2)\sigma.
$$
Since $\kappa_1=-f\cos\theta$ and $\partial_t=\lambda e_2$, we have $\partial_t f=0$. Then, $f|_N=f(s)$ and hence
\begin{eqnarray}
\kappa_1(s)&=&-\cos\theta f(s),\label{35}\\ 
\partial_s\kappa_2(s,t)&=&-\sin^{-1}\theta(\cos\theta f(s)+\kappa_2(s,t))(\cos\theta\kappa_2(s,t) +f(s) ) \label{36}.
\end{eqnarray}

If $M^3=\h^3$, then $f$ is necessarily a nonzero constant on $M^3$, say $f=f_0$ (see \cite[Theorem 4.2]{amc}) and the surface $N$ has one constant principal curvature (see Example \ref{hsur}), which is completely described in \cite{ag}. 

Therefore we have proved the following results.

\begin{corollary} \label{cs1}
Let $N$ be a constant angle surface in $\r^3$ with anti-torqued axis, where $\cos\theta\neq 0$. Then, the principal curvatures of $N$ satisfy \eqref{35} and \eqref{36}.
\end{corollary}

\begin{corollary} \label{cs2}
Let $N$ be a constant angle surface in $\h^3$ with anti-torqued axis, where $\cos\theta\neq 0$. Then, $N$ has one principal curvature constant and the other satisfies \eqref{36} with constant $f$.
\end{corollary}

\begin{remark}
If in particular $N$ is minimal in $\r^3$, then $\kappa_1(s)=-\kappa_2(s)=-\cos\theta f(s)$. Hence, solving \eqref{36} we obtain that $f|_N=f(s)=\frac{1}{2\cot\theta s+c}$ with $c\in \r$.
\end{remark}

We construct an example of constant angle surface in $\r^3$ with anti-torqued axis $\v$. Such an axis can be derived by using the position vector field $\Phi$ of $\r^m$. Then, $\v=\frac{\Phi}{|\Phi|}$ is a unit anti-torqued vector field on the connected Riemannian manifold $\r^m\setminus \{0\}$ (see \cite{dan}). 

\begin{example}
Let $\Phi$ be the position vector field on $\r^3$. Take the unit anti-torqued vector field $\v=\frac{\Phi}{|\Phi|}$ on $\r^3\setminus\{0\}$ whose conformal scalar is $f=\frac{1}{|\Phi|}$. The surfaces in $\r^3\setminus\{0\}$ whose normal vector field make constant angle with $\v$ are known as {\it constant slope surfaces}, introduced by Munteanu in \cite{mun}. These surfaces were completely described in the cited paper. Consider the constant slope surface $N\subset \r^3\setminus\{0\}$ parametrized by
$$
\phi(s,t)=s\sin\theta(\cos(\alpha(s))\cos(t),\cos(\alpha(s))\sin(t),\sin(\alpha(s))), \quad s>0,
$$
where $\alpha(s)=\cot \theta \log(s)$. Notice that the restriction of $f$ on $N$ is $f(s)=\frac{1}{s\sin(\theta)}$. The principal curvatures are 
$$
\kappa_1(s)=-\frac{\cot\theta}{s}=-\cos\theta f(s)
$$
and
$$
\kappa_2(s)=-\frac{1}{s}(\cot\theta+\tan(\cot\theta\log s)).
$$
It is clear that \eqref{35} holds for $\kappa_1$. After straightforward computations, we also have the verification of \eqref{36}. Finally, we observe that the first fundamental form of $\phi$ is given by $ds^2+(\sin\theta s\cos(\alpha(s)))^2dt^2,$ which represents a warped product structure. This follows from the fact that both principal curvatures  and the function $f\circ \phi$ depend only on $s$. Setting $\lambda(s)=\sin\theta s\cos(\alpha(s))$, we conclude that the warping function $\lambda $ satisfies the condition in \eqref{tw}.
\end{example}

We next present an example of constant angle surfaces in the upper-half space model of $\h^3$ with anti-torqued axis. Let $(x_1,...,x_m)$ be the canonical coordinates of $\r^m$. This model is given by $\h^m=\{(x_1,...,x_m):x_m>0\}$ endowed with the conformal metric $\langle,\rangle_h=\frac{1}{x_m^2}\langle , \rangle_e$, where $\langle , \rangle_e$ is the Euclidean metric. Then $\v=-x_m\partial_m$, where $\partial_m=\frac{\partial}{\partial x_m}$, is a unit anti-torqued vector field with conformal scalar $f=1$ (see \cite{amc}).

\begin{example}\label{hsur}
Let $(x,y,z)$ be the canonical coordinates of $\r^3$ and $N$ the right cone $z=k\sqrt{x^2+y^2}$ with $k>0$. If $U_e$ denotes the Euclidean unit normal to $N$, then, due to the conformal hyperbolic metric $\langle , \rangle_h$, $U_h=zU_e$ becomes the corresponding unit normal to $N$ in $\h^3$. Since $\langle U_e,\partial_z\rangle_e =\frac{1}{\sqrt{1+k^2}}$, it follows that 
$$\langle U_h,-z\partial_z \rangle_h=\frac{-1}{\sqrt{1+k^2}}=\cos\theta, \quad \theta \in [0,\pi],
$$ 
which implies that the cone $N$ is a constant angle surface in $\h^3$ with anti-torqued axis $\v=-z\partial_z$. Let $\kappa_i^e$ and $\kappa_i^h$ ($i=1,2$) denote the principal curvatures of $N$ in $\r^3$ and $\h^3$, respectively. Then, $\kappa_i^h=z\kappa_i^e+\langle U_e,\partial_z\rangle_e$ (see \cite{gss}), yielding that
$$
\kappa_1^h=-\cos\theta, \quad \kappa_2^h=\sqrt{1+k^2}.
$$
It is easy to see that the principal curvatures $\kappa_i^h$ satisfy \eqref{35} and \eqref{36}.
\end{example}

\subsection{Case $\mbox{dim}(M)>3$} Similar to Theorem \ref{ts1}, we aim to classify constant angle hypersurfaces $N$ in  Riemannian manifolds with anti-torqued axis. For this, we state the following preliminary result.

\begin{proposition} \label{phs1}
Let $N$ be a constant angle hypersurface in a complete Riemannian manifold $M$, $\textnormal{dim}(M)>3$, with anti-torqued axis $\v$, and let $T$ be the unit tangential component of $\v$ with respect to $N$. Then $N$ is either a ruled hypersurface, or admits an orthogonal decomposition $TN=\textnormal{span}\{T\}\oplus \mathcal{D}$, where the integral curves of $T$ are geodesics of $N$ and $ \mathcal{D} =\textnormal{span}\{T\}^\perp$ is an integrable distribution.
\end{proposition}

\begin{proof}
By the assumption, we have the decomposition given by \eqref{decomp}. From \eqref{34}, it follows $\nabla_TT=0$, i. e., $\mbox{span}\{T\}$ is an integrable distribution whose leaves are totally geodesic in $N$. If in addition $\cos\theta=0$, then using \eqref{33} and the Gauss formula we have $\nabla^0_TT=0$, which proves that $N$ is a ruled hypersurface in $M$. We next assume that $\cos\theta \neq 0$. Let $\mathcal{D}=\mbox{span}\{T\}^{\perp}$ be the orthogonal distribution to $\mbox{span}\{T\}$. Then, for $X,Y \in \mathcal{D}$, it follows from \eqref{31} that 
$$
 \sin \theta \langle \nabla_X T,Y \rangle -\cos \theta \langle A(X),Y\rangle =f\langle X,Y \rangle.
$$
On the other hand, because $\langle X  ,T\rangle=\langle Y  ,T\rangle=0$, we have
$$
\langle \nabla_X Y ,T\rangle-\langle \nabla_Y X ,T\rangle=-\langle \nabla_X T ,Y\rangle+\langle \nabla_Y T ,X\rangle.
$$
Together with the self-adjointness of $A$, this yields
$$
\langle [X, Y] ,T\rangle=0, \quad X,Y \in \mathcal{D}.
$$
Therefore, $\mathcal{D}$ is involutive, and by the Frobenius Theorem, it is integrable. 
\end{proof}

\begin{remark}
If $\mbox{dim}(N)=2$, due to \eqref{33}, the restriction of the shape operator $A$ on $\mathcal{D}=\mbox{span}\{T\}^{\perp}$, denoted by $A|_{\mathcal{D}}$, is trivially a multiple of the identity. Otherwise, $\mbox{dim}(N)>2$, this does not hold in general. However, as seen Example \ref{ehs}, we may impose such an extra condition in order to state a better classification result. 
\end{remark}

\begin{theorem}\label{ths1}
Let $N$ be a constant angle hypersurface in a complete Riemannian manifold $M$, $\textnormal{dim}(M)>3$, with anti-torqued axis $\v$. If $\mathcal{D}$ is the orthogonal distribution to $\textnormal{span}\{T\}$ and $A|_{\mathcal{D}}=\delta \textnormal{Id}$ with some function $\delta$, then $N$ is locally either a ruled surface or the twisted product of a real interval and a Riemannian manifold whose twisting function $\lambda$ satisfies $(\log \lambda)'=\cot\theta\delta-\frac{f}{\sin\theta}$, where $f$ is the conformal scalar of $\v$. If in addition $M=\r^m$ or $M=\h^m$, then $N$ is a warped product.
\end{theorem}

\begin{proof}
From Proposition \ref{phs1}, it follows that $N$ is either ruled surface or we have the orthogonal decomposition $TN=\mbox{span}\{T\}\oplus \mathcal{D}$, where $\mbox{span}\{T\}$ and $ \mathcal{D} =\textnormal{span}\{T\}^\perp$ are integrable distributions and $\nabla_TT=0$. Moreover, by the assumption we have $A(X)=\delta X$, $X\in \mathcal{D}$. From \eqref{31}, it follows that
$$
\nabla_X T=\mu X, \quad X\in \mathcal{D},
$$
where $\mu=\cot\theta\delta-\frac{f}{\sin\theta}$. This implies that each leaf of $\mathcal{D}$ is totally umbilical hypersurface of $N$ with the unit normal $T$ whose mean curvature of each leaf is $(1-m)\mu$. Moreover, the mean curvature vector field is trivally parallel, namely $\mathcal{D}$ is a spherical foliation. Consequently, by \cite[Proposition 3]{pr}, $N$ is locally the twisted product $I\times_\lambda F$, where $I$ is a real interval, $F$ is a Riemannian manifold, and $\lambda >0$ is a smooth function on $I\times F$. Hence, if $\langle,\rangle_N$ and $\langle,\rangle_F$ are metric tensors of $N$ and $F$, respectively, then we have
$$
\langle,\rangle_N=ds^2+\lambda^2 \langle,\rangle_F, \quad T=\partial_s.
$$
Since $\nabla_X \partial s=(\log \lambda)'X$, we have $(\log \lambda)'=\mu$. To complete the proof, by assuming $M=\r^m$ or $M=\h^m$ with $m>3$ we will show that the two functions $\delta$ and $f$ are constants on the leafs of $\mathcal{D}$. 
\begin{enumerate}
\item Because $\mathcal{D}$ is integrable, it follows that $\nabla_XY\in \mathcal{D}$ and hence
$$
(\nabla_X A)(Y)=X(\delta)Y+\delta\nabla_XY-A(\nabla_XY)=X(\delta)Y.
$$
By the Codazzi's equation, we obtain
$$
0=(\nabla_X A)(Y)-(\nabla_Y A)(X)=X(\delta)Y-Y(\delta)X, \quad X,Y\in \mathcal{D}.
$$
Since this holds for every $X\in \mathcal{D}$, we arrive at $X(\delta)=0$.
\item Denote by $R^0$ the Riemannian curvature tensor of $M$. Since $M$ is the Euclidean or hyperbolic ambient space, it follows that
$$
R^0(X,Y)\v=\epsilon(\langle Y,\v \rangle X-\langle X,\v \rangle Y), \quad \epsilon\in\{-1,0\}.
$$
For every $X,Y\in\mathcal{D}$, we have $R^0(X,Y)\v=0$, and $\nabla^0_X\v=fX$ and $\nabla_Y^0\v=fY$. This yields
\begin{eqnarray*}
0=R^0(X,Y)\v&=&X(f)Y+f\nabla^0_XY-Y(f)X-f\nabla^0_YX-f[X,Y]\\
&=&X(f)Y-Y(f)X, \quad X,Y\in \mathcal{D}.
\end{eqnarray*}
Similar to the previous item, we obtain $X(f)=0$, completing the proof.
\end{enumerate}
\end{proof}

We conclude this section by providing examples of constant angle hypersurface in $\r^m$ and $\h^m$ with anti-torqued axis.

\begin{example} \label{ehs}
\begin{enumerate}
\item {\it Euclidean space.} Let $\v=\frac{\Phi}{|\Phi|}$ be the anti-torqued vector field on $\r^4 \setminus\{0\}$, where $\Phi$ is the position vector field. Take a planar logarithmic spiral $(p(u),q(u))$ with $u\in \r$, where $p(u)=e^u\cos u$ and $q(u)=e^u\sin u$. Introduce $\zeta \in \s^2 \subset \r^3$ as 
$$
\zeta(v,w)=(\cos v\cos w,\cos v\sin w,\sin v), \quad v,w\in\r.
$$
We consider the rotational hypersurface $N$ parametrized by
\begin{equation}
\phi(u,\zeta)=(p(u)\zeta,q(u)), \quad \zeta \in \s^2. \label{rotpar}
\end{equation}
The unit normal to $N$ and the vector field $\v\circ \phi$ are given by
$$
U=\frac{1}{\sqrt{2}e^u}(-q'\zeta,p'), \quad \v=\frac{1}{e^u}(p\zeta,q).
$$
It is easy to see that the angle between $U$ and $\v$ is $\theta=\frac{3\pi}{4}$. Also we have 
$$
U=\frac{1}{\sqrt{2}}(-(\sin u+\cos u)\zeta,-\sin u+\cos u), \quad \v=(\cos u\zeta,\sin u)
$$
and hence $\v-\langle \v, U \rangle_e U=\frac{1}{2e^u}\phi_u$. We set the unit tangential component of $\v$ as follows
$$
T=\frac{\v-\langle \v, U \rangle_e U}{|\v-\langle \v, U \rangle_e U|_e}=\frac{1}{\sqrt{2}e^u}\phi_u=\frac{1}{\sqrt{2}e^u}(p'(u)\zeta,q'(u)),
$$
where $|\cdot|_e$ denotes the Euclidean norm. For any vector $\xi \in T_\zeta\s^2$ it follows that the vector $p(u)(\xi,0)\in T_{\phi(u,\zeta)}N$ is perpendicular to $T$. Therefore, the orthogonal distribution $\mathcal{D}$ to $\mbox{span} \{T\}$ is
$$
\mathcal{D}=\{p(u)(\xi,0)|\xi \in T_\zeta\s^2\}.
$$
The shape operator is
$$
A(T)=-\nabla^0_TU=\frac{1}{\sqrt{2}e^u}U_u=\frac{1}{\sqrt{2}e^u}T,
$$
which implies that \eqref{33} holds because $f\circ \phi =e^{-u}$ and $\cos\theta=-\frac{1}{\sqrt{2}}$. In addition, the restriction of $A$ on $\mathcal{D}$ is a multiple of the identity. In fact, we obtain
$$
A(p(u)(\xi,0))=-\nabla^0_{p(u)(\xi,0)}U=\frac{\sin(u)+\cos(u)}{\sqrt{2}}p(u)(\xi,0).
$$
On the other hand, the first fundamental form $\langle,\rangle_N$ of $N$ is given by
$$\langle,\rangle_N=2e^{2u}du^2+p(u)^2\langle,\rangle_{\s^2},$$
where $\langle,\rangle_{\s^2}$ is the metric tensor of $\s^2$. Changing the parameter $s=\sqrt{2}e^u$, we have 
$$
\lambda(s)=p(u(s))=\frac{s}{\sqrt{2}}\cos (\ln \frac{s}{\sqrt{2}}),
$$
which gives $\langle,\rangle_N=ds^2+\lambda(s)^2g_{\s^2}$. In other words, $N$ is a warped product $I\times_\lambda \s^2$, where $I$ is a real interval. With respect to this metric, $T=\partial_s$ and it is clear that the integral curves of $T$ are geodesics of $N$.
\item {\it Hyperbolic space}. Let $\v=-x_4\partial_4$ be the anti-torqued vector field on $\h^4$. We can consider a hypersurface $N$ parametrized by \eqref{rotpar}, where $p(u)=\cos\theta u$ and $q(u)=\sin\theta u$ with $\sin \theta \cos \theta\neq 0$ and $u>0$. Let $U^e$ and $U^h$ denote the Euclidean and hyperbolic unit normals, respectively. Using $U^e=(q'\zeta,-p')$ and $\langle U^e,\partial_4 \rangle_e =-\cos\theta $, we conclude that $U^h=(x_4\circ \phi)U_e=qU_e$ makes constant angle $\theta$ with $\v$, that is, $\langle U^h,\v \rangle_h =\cos\theta $. The unit tangential projection $T$ of $\v$ on $N$ is given by
$$
T=\frac{\v-\langle \v,U^h\rangle_h U^h}{|\v-\langle \v,U^h\rangle_h U^h|_h}=-q \phi_u.
$$
Let $\nabla^h$ and $\nabla^e$ denote the Levi-Civita connections on $\h^4$ and $\r^4$, respectively. Then, we have the relation (\cite{lo})
$$
\nabla^h_XY=\nabla^e_XY-\frac{1}{x_4}(\langle \partial_4,X\rangle_e Y+\langle \partial_4,Y\rangle_e X-\langle X,Y\rangle_e \partial_4).
$$
Hence, $\nabla^h_TU^h=\cos \theta T$. Since the conformal scalar $f$ of $\v$ is constantly equal to $1$, equation \eqref{33} holds. Similar to the previous item, the orthogonal distribution $\mathcal{D}$ to $\mbox{span} \{T\}$ is given by $\mathcal{D}=\{p(u)(\xi,0)|\xi \in T_\zeta\s^2\}.$ Then we have $\nabla^h_{p(u)(\xi,0)}U^h=(\sin\theta q+\cos\theta)p(u)(\xi,0)$, which means that the restriction of $A$ on $\mathcal{D}$ is a multiple of the identity. On the other hand, the first fundamental form of $N$ is given by
$$
\langle,\rangle_N=(\sin\theta u)^{-2}du^2+\cot^2\theta \langle,\rangle_{\s^2},
$$
where $\langle,\rangle_{\s^2}$ is the metric tensor of $\s^2$ in $\r^3$. By changing the parameter $s=\csc \theta \log u$, we have $\langle,\rangle_N=ds^2+\cot^2\theta \langle,\rangle_{\s^2}$, which is the metric tensor of the Riemannian product $I\times \s^2(\cot\theta)$, where $I$ is a real interval and $\s^2(\cot\theta)$ is the sphere in $\r^3$ of radius $\cot \theta$. This is the particular case of a warped product with constant warping function. With respect to this metric, $T=\partial_s$ and the integral curves of $T$ are geodesics of $N$.
\end{enumerate}
\end{example}

\section{Hypersurfaces with respect to torqued vector fields} \label{torqued}

Let $\v$ be a torqued vector field on a Riemannian manifold $M$ with conformal scalar $f$ and generating form $\omega$. By its definition $\v$ satisfies expression \eqref{int-tor} with $\omega(\v)=0$. If $\w$ is the vector field dual to $\omega$, then equivalently $\langle \w,\v\rangle=0$. In particular, the vector field $\v$ is concircular if $\omega=0$.

Let $N$ be an immersed hypersurface in $M$ by $\phi$, and let $U$ denote the unit normal of $\phi$. This last section is devoted to consider hypersurfaces satisfying
\begin{equation} \label{51}
\langle U, \v \rangle =\vartheta, \quad \vartheta \in \r,
\end{equation}
where $\v$ is a torqued vector field. If the tangential direction of $\v^{\top}$ on $N$ vanishes, then from \cite[Proposition 4. 7]{amc1} it follows that $N$ is totally umbilical. Such examples can be constructed as follows.

\begin{example}\label{5e1}
Let $I\subset \r$ be a real interval and let $u>0$ be a coordinate on $I$. Consider the warped product space $M=I\times_u \r^3$ enowed with the metric
$$
\langle , \rangle=du^2+u^2(dx^2+dy^2+dz^2).
$$ 
Denote by $\nabla^0$ the Levi-Civita connection on $M$.
\begin{enumerate}
\item Set $\v=u\partial_u$. From Proposition \ref{warp-con}, it follows that $\nabla^0_X \v=X$, that is, $\v$ is a concircular vector field on $M$ with conformal scalar $f=1$ (see also \cite[Example 1.1]{cbook}). Let $N$ be the fiber $N=\{u=u_0\in\r,u_0>0\}$. Its unit normal is $U=\partial_u$ and hence the function $\langle \v,U \rangle=u$ is the constant $u_0$ on $N$. It is easy to see that $N$ is totally umbilical.
\item Let $\v=ux\partial_u$. Then, $\v$ is a torqued vector field on $M$ (see \cite[Theorem 2.3]{c01}). Using Proposition \ref{warp-con}, we have
$$
\nabla^0_X \v=x X+d(\log x) \v, \quad X\in\x(M),
$$
where the conformal scalar and generating form of $\v$ are respectively $f=x$ and $\omega=d(\log x)$. It is easy to see that $\omega(\v)=0$, i.e. $\v$ is a torqued vector field. Assume that $N$ is the vertical hyperplane $N=\{x=x_0\in\r\}$, which is totally geodesic. Then its unit normal is $U=\frac1u\partial u$ and hence $\langle \v,U\rangle=x$ is a constant function on $N$.
\end{enumerate}
\end{example}

\begin{definition}\label{def5}
Let $\v$ be a torqued vector field globally defined on a Riemannian manifold $M$. An orientable hypersurface $N\subset M$ is called a {\it torqued hypersurface with axis} $\v$ if $\v^{\top} \neq 0$ and condition \eqref{51} is satisfied. 
\end{definition}

In particular, $N$ is called a non-trivial concircular hypersurface if the axis $\v$ in Definition \ref{def5} is a concircular vector field (see  \cite{loy4}). 

Assume that $N$ is a torqued hypersurface in $M$ with axis $\v$. Then we have the following orthogonal decompositions of $\v$ and its generative $\w$ in their tangential and normal components with respect to $N$
\begin{equation}\label{decompt}
\v=\alpha T_1  +\vartheta U, \quad \w=\beta T_2 +\gamma U,
\end{equation}
where $T_1,T_2 \in \mathfrak{X}(N)$ are unit vector fields and $\alpha,\beta,\gamma$ are smooth functions.

The functions given in \eqref{decompt} satisfy some properties.

\begin{proposition} \label{abg}
Let the functions $\alpha,\beta,\gamma$ and the constant $\vartheta$ be given as in \eqref{decompt}.
\begin{enumerate}
\item[(i)] $\alpha$ is not a constant function on $N$. 
\item[(ii)] If $\vartheta=0$, then $\beta$ is zero on $N$.
\item[(iii)] If $\v$ is not a concircular vector field, then $\gamma$ is never zero on $N$.
\end{enumerate}
\end{proposition}
\begin{proof}
Using $\langle \v,\w \rangle=0$, we have $\alpha\beta +\vartheta \gamma=0$. Since \cite[Proposition 3.2]{amc}, $\v$ is never of constant length, proving the first item. It means that $\gamma$ is never zero on $N$ as long as $\v$ is not a concircular vector field. The proof of the item (ii) is clear.
\end{proof}

\begin{remark}
Condition \eqref{51} does not imply that the torqued vector field $\v$ makes constant angle with the unit normal $U$ of the torqued hypersurface $N$. In fact, the angle between $\v$ and $U$ is given by $\frac{\vartheta}{\sqrt{\alpha^2+\vartheta^2}}$, which is non-constant due to Proposition \ref{abg}.
\end{remark}

Let $\nabla^0$ and $\nabla$ be the Levi-Civita connections on $M$ and $N$, respectively. Since $\v$ is torqued, for any $X\in \mathfrak{X}(N)$ it follows that
\begin{eqnarray*}
\nabla^0_X\v&=&fX+\beta\langle X,T_2\rangle (\alpha T_1  +\vartheta U), \\ \nabla^0_X\v&=&X(\alpha)T_1+\alpha\nabla_XT_1+\alpha\langle A(X),T_1\rangle U-\vartheta A(X).
\end{eqnarray*}
Comparing the tangent and normal components, we obtain
\begin{eqnarray}
fX+(\beta\langle X,T_2\rangle - X(\alpha))T_1&=&\alpha\nabla_XT_1-\vartheta A(X), \label{eqs-t1}\\ 
\vartheta\beta\langle X,T_2\rangle&=&\alpha\langle A(X),T_1\rangle.\label{eqs-t2}
\end{eqnarray}
\begin{proposition}\label{pro-torq}
Let $N$ be a torqued hypersurface in a Riemannian manifold $M$ with axis $\v$. Then, we have  \eqref{eqs-t1} and \eqref{eqs-t2}.
\end{proposition}

In particular, if $\vartheta=0$ then from Proposition \ref{abg} it follows that $\beta$ is zero on $N$. Hence, \eqref{eqs-t1} and \eqref{eqs-t2} reduce to
\begin{eqnarray}
fX- X(\alpha)T_1&=&\alpha\nabla_XT_1, \label{eqs-t11}\\ 
0&=&\alpha\langle A(X),T_1\rangle . \label{eqs-t21}
\end{eqnarray}

It follows from \eqref{eqs-t21} that $A(T_1)=0$. Writing $X=T_1$ in \eqref{eqs-t11}, we conclude that $f=T_1(\alpha)$ and $\nabla_{T_1}T_1=0$. This means that the integral curves of $T_1$ are geodesics of the ambient space $M$. Also, for every $Y \in \mbox{span}\{T\}^{\perp}$ we have
\begin{equation*}
fY-Y(\alpha)T_1=\alpha\nabla_YT_1, \quad Y \in \mbox{span}\{T\}^{\perp}, 
\end{equation*}
which yields that $Y(\alpha)=0$ and $\nabla_YT_1=\frac{f}{\alpha}Y$. Therefore, following the proof of Theorem \ref{ths1}, we have obtained the next result.

\begin{theorem}\label{ths10}
Let $N$ be a torqued hypersurface in a complete Riemannian manifold $M$ with axis $\v$, where $\vartheta=0$. Then $N$ is locally a ruled hypersurface which is the warped product of a real interval $I$ and a Riemannian manifold $F$, with the metric tensor
$$
\langle,\rangle_N=ds^2+\alpha^2(s)\langle,\rangle_F, \quad T=\partial_s,
$$
where the conformal scalar $f$ of $\v$ satisfies $f=\alpha'(s)$.
\end{theorem}

\begin{remark}
When $\v$ is a concircular vector field and the ambient space is a real space form, Theorem \ref{ths10} concludes the result \cite[Proposition 7]{loy4}.
\end{remark}

An example of such a torqued hypersurface can be constructed as follows.

\begin{example}
Let the ambient space $M$ be the warped product as given in Example \ref{5e1}. We consider the torqued vector field  $\v=F(x,y,z)u\partial_u$ on $M$. Then, its conformal scalar $f=F(x,y,z)$ and generating form $\omega=d(\log F)$. Let also $\Sigma \subset \r^3$ be a surface given in implicit form $F(x,y,z)=1$ with metric tensor $\langle,\rangle_\Sigma$ and unit normal $\widetilde{U}$. Assume that the hypersurface $N\subset M$ is the cylinder $I\times \Sigma$. Then, the lift of $\widetilde{U}$ on $\r^3$ is normal to $N$. Hence If $U$ is the unit normal vector field to $N$, then it follows that $U=\frac{1}{u|\nabla^eF|}\nabla^eF$, where $\nabla^eF$ is the gradient of $F$ in $\r^3$. Since $U$ has no $\partial_u$-component, we have $\langle U,\v\rangle=0$, which means that $\v\in \x(N)$. Comparing with \eqref{decompt}, it follows that $\alpha=uF(x,y,z)$, and the restriction of $\alpha$ on $N$ is $u$. It is direct to see that the first fundamental form of $N$ is given by $\langle,\rangle_N=du^2+u^2\langle,\rangle_\Sigma$, verifying Theorem \ref{ths10}.
\end{example}

Assume next that $\v$ is a concircular vector field and $N$ is a concircular hypersurface in $M$ with axis $\v$. In this case, it follows directly from \eqref{eqs-t1} and \eqref{eqs-t2} that 
\begin{equation*}
fX- X(\alpha)T_1=\alpha\nabla_XT_1-\theta A(X), \quad X\in \x(N), 
\end{equation*}
and $\langle A(X),T_1\rangle=0$ for $X\in \x(N)$, respectively. Hence, we again have that $A(T_1)=0$, $f=T_1(\alpha)$, and $\nabla_{T_1}T_1=0$, impyling the following result.

\begin{theorem}\label{ths11}
Let $N$ be a non-trivial concircular hypersurface in a complete Riemannian manifold $M$. Then $N$ is locally a ruled hypersurface.
\end{theorem}

Examples of concircular hypersurfaces in real space forms can be found in \cite{loy4}. We next provide examples of such hypersurfaces in an ambient space whose sectional curvature is not constant.

\begin{example}
Take the warped product space $M$ given Example in \ref{5e1} and the concircular vector field $\v=u\partial_u$ on $M$.

\begin{enumerate}
\item Case $\langle \v,U\rangle=0$. We use spherical coordinates $(r,\rho,\varphi)$ of $\r^3$ and consider the hypersurface $N\subset M$ parametrized by $\phi(u,\rho,\varphi)=(u,r_0,\rho,\varphi)$. Its unit normal is $U=\frac1u \partial_r$ and hence $\langle \v,U\rangle=0$. The unit tangential projection of $\v$ is $T=\partial_u$. From Proposition \ref{warp-con}, it follows that $A(T)=-\nabla^0_TU=0$ and $\nabla^0_TT=0$. Then, the hypersurface $N$ is ruled in $M$ whose fundamental form is given by the warped metric $\langle,\rangle_N=du^2+(r_0u)^2\langle,\rangle_{\s^2}$, where $\langle,\rangle_{\s^2}$ is the standart metric of $\s^2\subset \r^3$. This verifies Theorem \ref{ths10}.

\item Case $\langle \v,U\rangle=\vartheta \neq 0$. Consider the graph hypersurface $N$ in $M$ on the $(r\rho\varphi)$-space. A parametrization of $N$ is given by $\phi(r,\rho,\varphi)=(h(r,\rho,\varphi),r,\rho,\varphi)$. We particularly assume that $h(r,\rho,\varphi)=\vartheta\sec r$, where $\vartheta>0$ is some constant. The unit normal to $N$ is $U=\cos r(\partial_u- \frac{\sin r}{\vartheta} \partial_r)$ and then $\v\circ \phi=\vartheta\sec r \partial_u$. Hence, since $\langle U,\v\rangle=\vartheta$, $N\subset M$ becomes a concircular hypersurface. Moreover, the tangential projection is $\v^\top=\sin r(\vartheta\tan r \partial_u+\cos r \partial_r)$, and $T=\frac{\v^\top}{|\v^\top|}=\sin r\partial_u+\vartheta^{-1}\cos^2 r\partial_r$. If $\nabla^0$ denotes the Levi-Civita connection on $M$, then by using Proposition \ref{warp-con}, we obtain $\nabla^0_TT=0$, that is, $N$ is a ruled hypersurface.

\end{enumerate}
\end{example}

\section*{Acknowledgments}
This study was supported by Scientific and Technological Research Council of Turkey (TUBITAK) under the Grant Number (123F451). The authors thank to TUBITAK for their supports.

\section*{Declarations}

\begin{itemize}
\item Conflict of interest: The authors declare no conflict of interest.

\item Ethics approval: Not applicable.

\item Availability of data and materials: Not applicable.

\item Code availability: Not applicable.

\item Authors’ contributions: Each author contributes equally to the study.
\end{itemize}


 \end{document}